\documentclass[10pt, reqno]{amsart}
\usepackage{epsfig,latexsym,amsfonts,amssymb,amsmath,amscd,graphics,epic}
\usepackage{amsfonts,amssymb,amsmath,amscd,amsthm}
\usepackage[mathscr]{eucal}
\usepackage{mathrsfs}

\usepackage{mathbbol}

\usepackage{oldgerm,units}
\usepackage{wrapfig,epsfig}

\usepackage{ifthen}
\usepackage{mathbbol}

\usepackage{amsthm}
\usepackage[mathscr]{eucal}
\usepackage{mathrsfs}
\usepackage{mathbbol}
\usepackage{oldgerm,units}
\usepackage{wrapfig}

\newtheorem{theorem}{Theorem}[section]
\newtheorem{proposition}[theorem]{Proposition}
\newtheorem{definition}[theorem]{Definition}

\newtheorem{lemma}[theorem]{Lemma}

\newtheorem{notation}[theorem]{Notation}
\newtheorem{corollary}[theorem]{Corollary}

\newtheorem{example}[theorem]{Example}
\newtheorem{remark}[theorem]{Remark}
\newtheorem{iremark}[theorem]{(Important) Remark}

\newtheorem*{theorem*} {Theorem}
\newtheorem*{corollary*} {Corollary}
\newtheorem*{kquestion*} {Known question}
\newtheorem*{question*} {Question}
\newtheorem*{remark*}{Remark}
\newtheorem*{example*}{Example}
\newtheorem*{notation*}{Notation}

\newcommand{\chos}[2]{{#1 \choose #2}}

\newcommand{\trop}[1]{\mathcal{#1}}

\newcommand{\tG}{\trop{G}}
\newcommand{\tH}{\trop{H}}

\newcommand{\tQ}{\trop{Q}}

\newcommand{\tR}{\trop{R}}

\newcommand{\To}{\longrightarrow }
\newcommand{\iso}{\overset{\sim}{\longrightarrow}}

\newcommand{\al}{\alpha}

\newcommand{\Gm}{\Gamma}

\newcommand{\Dl}{\Delta}

\newcommand{\Lm}{\Lambda}

\newcommand{\rnk}[1]{\operatorname{rk}(#1)}

\newcommand{\Cl}[1]{#1^\bullet}

    \ifx\proof\undefined
    \newenvironment{proof}{
    \smallskip
    \noindent\emph{Proof.}}{\hfill\(\Box\)
    \bigskip
    } \fi

\newcommand{\ifdef}[3]{\ifthenelse{\equal{#1}{true}}{#2}{#3}}

\numberwithin{equation}{section}

\input xy
\xyoption{all}

\usepackage{rotating}
\usepackage{amssymb,  mathrsfs, color,  rotating}
\usepackage{mathabx}
\input amssymb.sty
\def\semiring{semiring}

\def\nook{{c-rank}}

\def\nookind{{c-independent}}
\def\prll{\, {\|} \, }

\newcommand\ul[1]{\underline{#1}}
\newcommand\st[1]{\{#1\}}

\newcommand\nk[1]{\operatorname{c-rk}(#1)}

\newcommand\DM[1]{\operatorname{DM}(#1)}
\newcommand\UC[1]{\operatorname{UC}(#1)}
\newcommand\Hs[1]{\operatorname{Hs}(#1)}

\newcommand\Lat[1]{\operatorname{Lat}({#1})}

\newcommand\clos{{\operatorname{cl}}}

\def\tGz{{\tG_0}}

\def\tTzB{\{0,1\}}

\def\tlW{\widetilde W}
\def\tlell{\widetilde \ell}
\def\tlm{\widetilde m}

\def\ltw{0.5\textwidth}

\def\beginA{\pSkip \begin{center} \begin{minipage}{\ltw}}
\def\endA{\end{minipage} \end{center} \pSkip}

\def\htvrp{\widehat{\vrp}}
\def\chvrp{\widecheck{\vrp}}
\def\ivrp{{\vrp}^{-1}}
\def\ipsi{{\psi}^{-1}}

\newcommand\pup[1]{#1^\uarr}
\newcommand\pdn[1]{#1^\darr}
\def\darr{\downarrow} \def\uarr{\uparrow}

\newcommand\cl[2]{\rwcl{#1}{\str}{#2}}

\newcommand\rwcl[3]{#1[#2,#3]}
\def\str{\, \ast \,}

\newcommand{\ds}[1]{\ {#1} \ }
\newcommand{\dls}[1]{\; {#1} \; }
\newcommand{\dss}[1]{\quad {#1} \quad }
\def\iff{\Leftrightarrow}

\def\tlL{\widetilde L}

\def\gvrp{\htvrp}
\def\lvrp{\chvrp}

\def\sji{\textsf{sji}}
\def\ji{\textsf{ji}}
\def\smi{\textsf{smi}}
\def\mi{\textsf{mi}}

\def\smiNoT{\#_{\smi \neq T}}
\def\smiNo{\#_{\smi}}
\def\sjiNoB{\#_{\sji \neq B}}
\def\sjiNo{\#_{\sji}}

\def\miNoT{\#_{\mi \neq T}}

\def\jiNoB{\#_{\ji \neq B}}

\def\ss{\sigma}
\def\cc{\kappa}
\def\Dl{\Delta}
\def\pSkip{\vskip 1.5mm \noindent}
\def\ltw{0.7\textwidth}
\newcommand\HH{\mathscr{H}}

\newcommand\hgt[1]{ \operatorname{ht}( #1)}
\def\iso{ \operatorname{iso}}

\def\spec{\operatorname{spec}}
\def\LAT{\operatorname{LAT}}
\def\SLat{\operatorname{SLAT}}
\def\FLat{\operatorname{FLAT}}
\def\FSLat{\operatorname{FSLAT}}

\def\ICM{\operatorname{ICM}}

\def\vrp{\varphi}

\def\dlL{L ^{*}}
\def\dlM{M ^{*}}

\def\Pos{P}
\newcommand\rvs[1]{{#1}_{\operatorname{rvs}}}
\newcommand\cmp[1]{{#1}^{\operatorname{c}}}
\newcommand\trn[1]{{#1}^{\operatorname{t}}}

\newcommand\boxtext[1]{\pSkip \qquad \qquad \qquad \framebox{\parbox{\ltw}{#1}}\pSkip}

\def\Pow{\operatorname{Pw}}
\def\iff{\Leftrightarrow}
\def\imp{\Rightarrow}

\def\adj{\operatorname{adj}}

\newcommand\adjsup[1]{#1_{\adj-\sup}}

\def\H{\HH}

\def\Cl{\operatorname{Col}}
\def\Rw{\operatorname{Row}}

\def\sm{\setminus}

\def\1{1^\nu}
\def\0{0^\nu}

\newcommand{\etype}[1]{\renewcommand{\labelenumi}{(#1{enumi})}}
\def\eroman{\etype{\roman}}
\def\ealph{\etype{\alph}}

\def\({\left(}
\def\){\right)}

\def\tGz{{\tG_0}}

\def\onto{\twoheadrightarrow}

\def\bool{\mathbb B}

\def\sbool{{\mathbb{SB}}}

\newcommand{\per}[1]{\operatorname{per}({#1})}

\begin{document}


\title[C-independence and c-rank of Posets and Lattices]
{C-independence and c-rank of Posets and Lattices}


\author{Zur Izhakian}

\address{School of Mathematical Sciences, Tel Aviv
    University, Ramat Aviv,  Tel Aviv 69978, Israel \vskip 1pt
Department of Mathematics, Bar-Ilan University, Ramat-Gan 52900,
Israel} \email{zzur@post.tau.ac.il;zzur@math.biu.ac.il}

\thanks{The research of the first author has  been  supported  by the
Israel Science Foundation (grant No.  448/09) and  by the
\textit{Oberwolfach Leibniz Fellows Programme (OWLF)},
Mathematisches Forschungsinstitut Oberwolfach, Germany.}

\thanks{The second author gratefully acknowledges the hospitality of the
{Mathematisches Forschungsinstut Oberwolfach} during a visit to
Oberwolfach.}

\author{John Rhodes}
\address{Department of Mathematics, University of California, Berkeley,
970 Evans Hall \#3840, Berkeley, CA 94720-3840 USA }
\email{rhodes@math.berkeley.edu;blvdbastille@aol.com}


\date{\today }





\begin{abstract}
Continuing with the authors concept (and results) of defining
independence for columns of a boolean and superboolean matrix, we
apply this theory to finite lattices and finite posets,
introducing boolean and superboolean matrix representations for
these objects. These representations yield the new concept of
c-independent subsets of lattices and posets, for which the notion
of c-rank is determined as the cardinality of the largest
c-independent subset. We characterize this c-rank and show that
c-independent subsets have a very natural interpretation in term
of the maximal chains of the Hasse diagram and the associated
partitions of the lattice. This realization has direct important
connections with chamber systems.
\end{abstract}

\maketitle


\section*{Introduction}

The concept of boolean and superboolean representations had been
introduced first in \cite{IJmat} for finite hereditary
collections, and later was studied in depth for matroids
\cite{IJmatII}. In the present paper we broaden this concept to
finite partially ordered sets (written posets, as usual) and
mainly to finite lattices. Furthermore, we show that the same
representation ideas are naturally applicable to structured sets,
either finite or infinite (cf.~ \S\ref{sec:2}).

Our representations are performed by using matrices with
coefficients  over the superboolean semiring~\cite{IJmat}, a
certain instance of a finite supertropical semiring
\cite{zur05TropicalAlgebra,IzhakianRowen2007SuperTropical}. The
algebra of these matrices provides a proper notion of linear
independence \cite{IzhakianTropicalRank}, without the use of
negation, which is absent in the ``weak'' structure of semirings.
This notion of independence, determined for lattices and posets
via their representations, is at the heart of our theory and leads
neutrally to the introduction of the  \nook, defined to be the
cardinality of the largest independent subset.

When dealing with lattices, independent subsets have a fundamental
correspondence with the maximal chains of the lattice. In
particular, we prove that the \nook \ of a finite lattice equals
its height (Theorem~ \ref{thm:4.9}). Introducing the idea of
``pushing chains'' (cf. Definition \ref{def:push}), we show that
this pushing operation preserves lattice independent subsets.


The perspective of the Hasse diagram, together with that of
Dedekind-MacNeille completion, leads us to partition of lattices
-- a novel idea -- which plays a major role in our theory. Another
important notion in our theory is that of  {``partial cross
sections''}, which provides  a characterization of properties of a
subset with respect to a certain partition. All these enable us to
determine the fundamental connection between independence of
lattice subsets (determined by its matrix representation) and the
actual lattice structure (Theorem \ref{thm:pcs}).

Finally, assisted by boolean modules and their corresponding
lattices, we apply our representation techniques to finite
hereditary collections (also known as finite abstract simplical
complexes), yielding additional connections between the boolean
representation of these objects and finite lattices (Theorems~
\ref{thm:5.3} and \ref{thm:5.4}).

\begin{notation*}\label{nott:set}
 In this paper, for simplicity, we use the
following notation: Given a subset $X \subseteq E$, and elements
$x \in X$ and $p \in E$, we write $X- x$ and $X +y$ for $X \sm \{
x\}$ and $X \cup \{ y \}$, respectively; accordingly we write
$X-x+y$ for
 $(X\sm \{x\}) \cup \{ y\}$.
%
\end{notation*}

\section{Boolean and superboolean algebra }

The very well known  \textbf{boolean \semiring} is the two element
idempotent \semiring \ $\bool := (\{ 0,1\}, + , \cdot \;),$ whose
addition and multiplication are given respectively by the
following tables:
$$ \begin{array}{l|ll}
   + & 0 & 1  \\ \hline
     0 & 0 & 1  \\
     1 & 1 & 1 \\
   \end{array} \qquad \text{and} \qquad
\begin{array}{l|ll}
   \cdot & 0 & 1\\ \hline
     0 & 0 & 0  \\
     1 &  0 & 1  \\
   \end{array} \ .
   $$

The \textbf{superboolean \semiring}  $\sbool : = (\{ 1, 0, 1^\nu
\},  + , \cdot
 \;)$  is three element supertropical \semiring \ \cite{IzhakianRowen2007SuperTropical}, a
``cover'' of the boolean \semiring, endowed with the two binary
operations:
$$ \begin{array}{l|lll}
   + & 0 & 1 & 1^\nu \\ \hline
     0 & 0 & 1 & 1^\nu \\
     1 & 1 & 1^\nu & 1^\nu \\
     1^\nu &1^\nu & 1^\nu & 1^\nu \\
   \end{array} \qquad
\begin{array}{l|lll}
   \cdot & 0 & 1 & 1^\nu \\ \hline
     0 & 0 & 0 & 0 \\
     1 &  0 & 1 & 1^\nu \\
     1^\nu & 0 & 1^\nu & 1^\nu \\
   \end{array}
   $$
 addition and multiplication, respectively.
This semiring is totally  ordered by $ 1^\nu \
> \ 1 \
> \ 0 .$
Note that $\sbool$ is \textbf{not} an idempotent \semiring,  since
$1 +1 = \1$, and thus  $\bool$ is \textbf{not} a subsemiring of
$\sbool$. The element~$\1$ is called the \textbf{ghost} element,
where $\tGz := \{0, \1\}$ is the  \textbf{ghost ideal}\footnote{In
the supertropical setting, the elements of the complement of
$\tGz$ are called \textbf{tangibles}.}  of $\sbool$.


\subsection{Boolean matrices}\label{ssec:boolMat}

The semiring $M_n(\sbool)$ of $n \times n$ superboolean matrices
with entries in ~$\sbool$ is defined in the standard way, where
addition and multiplication are induced from the operations of
$\sbool$ as in the familiar matrix construction. The unit element
$I$ of
 $M_n(\sbool)$, is the matrix with $1$ on the main diagonal and
whose off-diagonal entries are all $0$.

A typical matrix is often denoted as $A =(a_{i,j})$, and  the zero
matrix is  written as $(0)$. A matrix is said to be a
\textbf{ghost matrix} if all of its entries are in $\tGz$.  A
boolean matrix is a matrix with coefficients in $\tTzB$, the
subset of boolean matrices is denoted by $M_n(\bool).$

The following discussion is presented for superboolean matrices,
where boolean matrices are considered as superboolean matrices
with entries in $\tTzB$.
 Note that boolean
matrices $M_n (\bool)$ are \textbf{not} a sub-semiring of the
semiring of superboolean matrices $M_n(\sbool)$.

In the standard way, for any matrix $A \in M_n(\sbool)$, we define
the \textbf{permanent} of  $A = (a_{i,j})$ as:
\begin{equation}\label{eq:det1}
\per{A} := \sum _{\pi \in S_n}a_{\pi (1),1} \cdots a_{\pi
(n),n}\end{equation}
where $S_n$ stands for the group of permutations  of
$\{1,\dots,n\}$. Note that the permanent of a boolean matrix can
be $\1$.           We say that a matrix $A$ is
\textbf{nonsingular} if $\per{A} = 1$, otherwise $A$ is said to be
\textbf{singular}.

\begin{lemma}[{\cite[Lemma 3.2]{IJmat}}]\label{lem:2.1.f}
A matrix $ A  \in M_n(\sbool)$ is nonsingular iff by independently
permuting columns and rows it has the triangular form
\begin{equation}\label{eq:trgform}
  A' :=  \(\begin{array}{cccc}
  1 & 0 &\cdots & 0 \\
  * & \ddots &  \ddots &\vdots\\
\vdots & \ddots & 1& 0\\
  * &  \cdots &*  & 1\\
   \end{array}\),
\end{equation}
 with all diagonal entries $1$, all entries above the diagonal are
 $0$, and the entries below the diagonal belong to  $\{1, \1, 0 \}$.

  Such
 reordering of $A$ is equivalent to multiplying the matrix $A$ by two
 permutation matrices $\Pi_1$ and~ $\Pi_2$ on the right and on the
 left, respectively, i.e., $A' := \Pi_1 A \Pi_2$.
\end{lemma}

Let $A$ be an $m \times n$ superboolean matrix. We say that a $k
\times \ell$ matrix $B$, with $k \leq m$ and $\ell \leq n$,  is a
\textbf{submatrix} of~$A$ if $B$ can be obtained by deleting rows
and columns of~$A$. In particular, a \textbf{row} of a matrix $A$
is an $1 \times n$ submatrix of $A$, where a \textbf{subrow} of
$A$ is an $1 \times \ell$ submatrix of $A$, with  $\ell \leq n$. A
\textbf{minor} is a submatrix obtained by deleting exactly one row
and one column of a square matrix.

\begin{definition}[{\cite[Definition 3.3]{IJmat}}]\label{def:marker}
A \textbf{marker} $\rho$ in a matrix is a subrow having a single
$1$-entry and all whose other entries are $0$; the length of
$\rho$ is the number of its entries. A marker of length $k$ is
written $k$-marker.
\end{definition}
For example the nonsingular matrix $A'$ in \eqref{eq:trgform} has
a $k$-marker for each $k = 1,\dots, n$, appearing in this order
from bottom to top. (Note that in general markers need not be
disjoint.)

\begin{corollary}[{\cite[Corollary 3.4]{IJmat}}]\label{cor:nonsing}
If a matrix $ A  \in M_n(\sbool)$ is a nonsingular matrix, then
$A$ has an
$n$-marker. 
\end{corollary}

\begin{definition}[{\cite[Definition 1.2]{IzhakianTropicalRank}}]\label{def:tropicDep}
A collection of vectors  $v_1,\dots,v_m \in \sbool^{(n)}$ is said
to be \textbf{dependent}
 if there exist $\al_1,\dots,\al_m \in \tTzB$,  not all of them $0$,
 for which
$$
   \al_1 v_1 +  \cdots + \al_m  v_m \in
  \tGz^{(n)}.
$$
Otherwise the vectors are said to be \textbf{independent}.
\end{definition}

The \textbf{column rank} of a superboolean matrix $A$ is defined
to be the maximal number of independent columns of $A$. The
\textbf{row rank} is defined similarly with respect to the rows of
$A$.


\begin{theorem}[{\cite[Theorem
3.11]{IzhakianTropicalRank}}]\label{thm:rnkSing} For any
supertropical matrix $A$ the row rank and the column rank are the
same, and this rank is equal to the size of the maximal
nonsingular submatrix of $A$.
\end{theorem}


\begin{definition} Let  $A = (a_{i,j})$ be a superboolean matrix.
The \textbf{complement}  $\cmp{A} := (\cmp{a_{i,j}})$ of $A$ is
defined by the role $ \cmp{a_{i,j}} = 1 \Leftrightarrow  a_{i,j} =
0$, and  $ \cmp{a_{i,j}} = \1 $ for every ghost  entry $ {a_{i,j}}
= \1$. The \textbf{transpose} $\trn{A} = (\trn{a_{i,j}})$ of $A$
is given by $ \trn{a_{i,j}}   = a_{j,i} .$
\end{definition}

Then we can conclude the following:
\begin{corollary}\label{cor:nRankTran}
The rank of a superboolean matrix is invariant under
\begin{enumerate}\eroman

   \item permuting of rows (columns); \pSkip

   \item deletion of a row (column) whose entries are all in $\tGz$; \pSkip

   \item deletion of a repeated  row or column; \pSkip

   \item  transposition, i.e., $\rnk{A} = \rnk{\trn{A}}$.
\end{enumerate}
\end{corollary}
\begin{proof}
Immediate by Theorem \ref{thm:rnkSing}.
\end{proof}

\begin{proposition} Transposition and complement commute, i.e.,
$\cmp{(\trn{A})} = \trn{(\cmp {A})}$  for any superboolean matrix
$A \in M_n(\sbool)$.
\end{proposition}
\begin{proof} Straightforward: $\cmp{(\trn{a_{i,j}})} = \cmp{({a_{j,i}})} =  \trn{(\cmp{{a_{i,j}}})}.$
\end{proof}

\begin{notation}\label{nott}
Given a matrix $A$ and a subset $Y \subseteq \Cl(A)$ of columns of
$A$, we write $\cl{A}{Y}$ for the submatrix of $A$ having the
columns $Y$. Sometimes we refer to $\Cl(A)$ as a collection of
vectors, but no confusion should arise. Given also a subset $X
\subseteq \Rw(A)$ of rows of $A$, we define $\rwcl{A}{X}{Y}$ to be
the submatrix of $A$ having the intersection of columns $Y$ and
the rows $X$, often also referred to as a collection of
sub-vectors.
\end{notation}

%
%
%
%
%
%
%

\section{Abstract setting}\label{sec:2}

\subsection{Structured sets}
Let $X$ be a nonempty finite set, i.e., $|X| = n$, and let $\tR$
be a binary relation defined on  the elements of $X$, written $x_i
\dls \tR  x_j$;  thus $\tR$ determines a structure on $X$. We
denote such a pair by $(X, \tR),$ and call it a \textbf{structured
set} (over $X$).


Given a structured set $(X, \tR)$,   with $X := \{ x_1, \dots, x_n
\} $ a finite set of elements, we associate $X$ with the  $n
\times n $ boolean matrix $A(X) := (a_{i,j})$, called
\textbf{structure matrix}, defined as
\begin{equation}\label{eq:matRelation}
 a_{i,j} := \left \{
\begin{array}{ll}
  1 & \text{if }  x_i \ds \tR x_j, \\[1mm]
  0 & \text{otherwise}. \\
\end{array}
\right.
\end{equation}
We write $(X,A)$ for the set $X$ and together with its structure
matrix $A = A(X)$ defined above, and  call this pair again a
\textbf{structured set}.

Having the above construction it is clear that the relation $\tR$
is fully recorded by the matrix $A$ and vise versa. Therefore, we
identify the relation $\tR$ on $X$ with the matrix $A := A(X)$.

\subsection{Independence}


we open with the key definition of our further development:

\begin{definition}\label{def:nook}Given a structured set $X:=(X, A)$  we define the \textbf{\nook} of $X$ as
$$ \nk{X} := \rnk{\cmp{A}}, \qquad  \cmp{A} :
= \cmp{(A(X))}.$$
\end{definition}

Given a structured set $X := (X,A)$, consider the matrix $
\cmp{A}$, written also as $\cmp{A} := \cmp{A}(X)$.  We say that a
subset $W \subseteq X$ is \textbf{\nookind} if the columns
$\cl{\cmp{A}}{W}$ of $\cmp{A}$ corresponding to $W$ are
independent in the sense of Definition \ref{def:tropicDep}. When
$|W| = k$, these columns contains a $k \times k $ nonsingular
submatrix $\rwcl{\cmp{A}}{U}{W}$ with $U \subseteq X$ and  $|U| =
k$ (cf. Theorem \ref{thm:rnkSing}), which we call a
\textbf{witness} of $W$ (in $\cmp{A}$). Abusing terminology, we
also say that $U$ is a witness of $W$ in the set $X$. Permuting
independently the columns of a witness, it has the triangular Form
\eqref{eq:trgform}, cf. Lemma \ref{lem:2.1.f}.

Accordingly, suppose $X := (X,A)$,  $|X| = n$, is a structured set
and let $W \subseteq X$ be an independent subset with $|W| = k$.
Then we have the following properties satisfied:
\begin{enumerate} \ealph
    \item $\nk{W} = k \leq n$, \pSkip
    \item $\nk {W} \leq \nk{X}$, \pSkip
    \item $ \nk{X} \leq n$.
\end{enumerate}

\section{Finite lattices and finite boolean modules}

In this section we start an explicit study of certain classes of
structure sets, equipped with extra properties.

\subsection{Posets}
A major example for a structure set, and the most abstract in this
paper, is given by the following well known definition
\cite{Lattice,qtheory}.
\begin{definition}\label{def:poset}
A \textbf{partial order} is a binary relation $\leq$ over a set
$X$ which is reflexive, antisymmetric, and transitive, i.e., for
all $a, b, c \in X$, we have that:
\begin{enumerate}\eroman
    \item $a \leq a$ (reflexivity); \pSkip
    \item if $a \leq b$ and $b \leq a$ then $a = b$
    (antisymmetry); \pSkip
    \item if $a \leq b$ and $b \leq c$ then $a \leq c$ (transitivity).
\end{enumerate}
A pair $(\Pos, \leq)$, with $\leq$  a partial order,  is called a
partially order set -- \textbf{poset} for short.

The \textbf{reverse poset} $\rvs{\Pos} := (\Pos, \geq) $ of $\Pos
:= (\Pos, \leq)$ is defined by reversing the order of $\Pos$,
i.e., $$ p \leq q  \text{ in  } P  \dss {\iff} p \geq q \text{ in
} \rvs{P}.
$$
\end{definition}

\begin{definition}\label{def:upsetPos} Given a poset element $p \in
P$, we  define the \textbf{up-set} $\pup{p}$ and the
\textbf{down-set} $\pdn{p}$  of $p$ respectively as
$$ \pup{p} := \{ x \in P \ds | x \geq p \},
\dss{\text{and}} \pdn{p}  := \{ x \in P \ds | x \leq p \}.$$
A subset $I \subseteq P$  of a poset $P := (P, \leq)$ is an
\textbf{order ideal} if for
$$ p \in I \text{ and } q \leq p \dss \Rightarrow q \in I.$$
\end{definition}
Accordingly,  for each $p \in P$, the down-set $\pdn{p}$ is an
order ideal of $P$.

\begin{remark}\label{rmk:posetMat} The structure matrix $A(\Pos) := (a_{i,j})$ of
a poset $\Pos:=(\Pos, \leq)$, cf. \eqref{eq:matRelation}, is given
by
\begin{equation}\label{eq:matRelationPos}
 a_{i,j} := \left \{
\begin{array}{ll}
  1 & \text{if }  p_i \leq   p_j, \\[1mm]
  0 & \text{otherwise},  \\
\end{array}
\right. \end{equation} and therefore has the proprieties
\begin{enumerate} \ealph
    \item $a_{i,i} = 1$, by
    reflexivity, for every $i = 1,\dots, n$. \pSkip
    \item $a_{i,j} = 1$ iff $a_{j,i} =0$, by antisymmetry,  for any $i \neq j.$
\end{enumerate}
\end{remark}

\begin{proposition}\label{prop:rposet} The reversing $\rvs{P}$ of a poset $P := (P, \leq)$,
recorded by $(P, A),$ is equivalent to $(P, \trn{A})$.
\end{proposition}

\begin{proof} Obtained immediately by Remark \ref{rmk:posetMat}.
\end{proof}

\begin{corollary}\label{cor:reversing} Given a poset $P := (P, \leq)$, reversing the order
on $P$ does not change the rank of~$P$, that is $\nk{P } = \nk
{\rvs{P}}$.
\end{corollary}

\begin{proof}
Clear from Proposition \ref{prop:rposet} and Corollary
\ref{cor:nRankTran}.
\end{proof}

We recall an additional known non-negative function on posets (the
{\nook} was one of them, defined earlier in general for structured
set).

\begin{definition}\label{def:posetHgt} The
\textbf{height} of a poset $P := (P, \leq)$, written $\hgt{P}$, is
defined to be the length of the longest strict chain it contains,
i.e.,
$$ \hgt{ P}  := \max \{ \  k  \ds |  p_0 < p_ 2 < \cdots < p_k, \quad p_0 , p_ 2
, \dots, p_k \in P \}. $$
\end{definition}

\subsection{Semilattices} Let us recall some standard definitions \cite{Lattice,qtheory}.
\begin{definition}\label{def:semilattice} A (finite) poset $P:= (P, \leq)$ whose elements admit  a
\textbf{join} relation (also known as the least upper bound, or
the supremum), i.e., $p_i \vee  p_j$ for all $p_i, p_j \in P$, is
called a (finite, join) \textbf{semilattice}, denoted as $S:= (S,
\leq)$.
\end{definition}
We define the category $\SLat$ of semilattices, whose maps are
sup-maps, given as follows ($\bigvee X $ stands for the common
join  of the members of $X$):
\begin{definition}\label{def:supmap}
A semilattice map
$$ \vrp: (S, \leq) \To (S',\leq)$$
that satisfies $\vrp(\bigvee X) = \bigvee(\vrp(X))$  for all $X
\subseteq S$ is called a \textbf{sup-map}.
\end{definition}

Any  semilattice $(S, \leq)$ can be viewed as a   semigroup $(S,
+)$ by defining $$ s + t := t \vee s, \qquad \text{for any } s,t
\in S.$$ This semigroup is commutative ($ s + t = t + s$) and
idempotent ($s + s = s$ for any $s \in S$). When a (join)
semilattice $S$ has a bottom element $B := 0$ then $(S, +)$  is an
idempotent commutative monoid with unit~$0$, i.e., $s + 0 = 0+s =
s$ for every $s \in S$. We denote this monoid as  $(S, + , 0 )$.

Conversely given an idempotent commutative monoid $S:=(S, +, 0)$
with unit $0$, one can define the semilattice $(S, \leq)$ having
the role
$$ s \leq t  \quad \Leftrightarrow \quad s +r =t \ \text{for some } r \in S.  $$
This role gives a proper  poset. Indeed, $s \leq s$ (reflexivity)
since $s + 0 = s$ for every $s$, and since $s_1 \leq s_2  \leq
s_3$ iff there exist $r_1, r_2$ such that $s_1 + r_1 = s_2$ and
$s_2 + r_2 = s_3$ which implies $s_1 + r_1 + r_2 = s_2 + r_2 =
s_3$, thus $s_1 \leq s_3$ (transitivity). Finally (antisymmetry),
$s \leq t \leq s$ iff there exist $r_1, r_2 \in S$ such that $s +
r_1 = t$ and $t +r_2 = s$, but then
$$ s + t = s + (s + r_1) = s + r_1 = t, \qquad s + t = (t + r_2) + t = t + r_2 = s,  $$
implying $s =t$.

Since $S$ is a monoid, $s +t$ always exists, and   thus we define
$$ s \vee t = p := \min \{ q \in S  \ds |  q \geq s,t\}.
$$ Moreover, since $p = s + x = t + y$  for some $x,y \in S$, then
$$ p =  p + p = s + t + x +y  \geq s + t.$$

Having the above construction, we see that the category $\ICM$ of
idempotent commutative monoids, whose maps are monoid homomorphism
(i.e., $\phi: S \to S'$ such that  $\phi(s+t) = \phi(s) + \phi(t)$
for any $s,t \in S$), is isomorphic to the category of $\SLat$
whose objects are semilattices and its maps are sup-maps.

\subsection{Lattices} We open again with a familiar definition.

\begin{definition}\label{def:lattice }
A (finite) poset $(P, \leq)$ in which each pair of elements $p_i,
p_j$ admits a \textbf{join}  $p_i \vee p_j$ (also known as the
least upper bound, or the supremum) and a \textbf{meet} $p_i
\wedge p_j$ (also known as the greatest lower bound, or the
infimum)  is a (finite) \textbf{lattice}, written  $L := (L,
\leq)$.
\end{definition}

Given a subset $X \subseteq L$, $X := \{x_1, \dots, x_m \}$, we
write
$$ \bigvee X := x_1 \vee x_2 \vee \cdots \vee x_m, \qquad
\bigwedge X := x_1 \wedge x_2 \wedge \cdots \wedge x_m,$$
respectively for the common join and meet of the members of $X$.

 A poset $(P, \leq)$ which has a join for each
pair of elements and a (global) unique minimal element $B$, i.e.,
$B \leq p$ for all $p \in P$, called \textbf{bottom element}, is
also a lattice, where the meet is defined  by
$$ p_i\wedge p_j := \bigvee \{ q \in P \ds | q \leq p_i, p_j\}, \qquad \forall p_i, p_j \in P. $$
Note that $X := \{ q \in P \ds | q \leq p_i, p_j\} $ is nonempty
since $B \in X$, formally we define $\bigvee \emptyset := B $.
When a lattice $L$ has a unique maximal element, we call this
element the \textbf{top element} of $L$, and denote it $T$.

A lattice $(L, \leq)$  is \textbf{distributive} if $$s \wedge (t
\vee  t') = (s \wedge t ) \vee (s \wedge t')$$ for all $s,t,t' \in
L$. It is not difficult to show that this condition is equivalent
to the  dual condition $$s \vee (t \wedge  t') = (s \vee t )
\wedge (s \vee t').$$

A lattice $(L, \leq)$  is \textbf{complete} if for every subset $X
\subseteq L$, the join $\bigvee X$ and the meet  $\bigwedge X$
exist, where for $X = \emptyset$ we set $$\bigvee \emptyset := B
\dss {\text{and}} \bigwedge \emptyset := T.$$ The \textbf{dual
lattice} $\dlL := (\dlL, \leq)$  is defined over the same set of
elements of $L$ having the reversed order~$\leq$, i.e., $\dlL =
\rvs{L}$ (cf. Definition \ref{def:poset}).

A lattice map
$$ \vrp: (L, \leq) \To (L',\leq)$$
that satisfies $\vrp(\bigvee X) = \bigvee(\vrp(X))$  for all $X
\subseteq L$ is called as before  \textbf{sup-map}.  Taking $X=
\emptyset$, this implies $\vrp(B) = B'$. Similarly, $\vrp$ is
called an \textbf{inf-map} if $\vrp(\bigwedge X) =
\bigwedge(\vrp(X))$ for all $X \subseteq L$, where now $\bigwedge
\emptyset = T$, and thus $\vrp(T) = T'$. A map which is both
sup-map and inf-map is termed \textbf{sup-inf-map}.

We set $\LAT$ to be the category of  lattices whose map are
sup-maps, which  is a full subcategory of the category $\SLat$ of
semilattices. Both have the full subcategories $\FSLat$ a and
$\FLat$ of finite semilattices and finite lattices, respectively,
whose maps are sup-maps as well.

The reader should be note that the sup-maps preserve  the
structure of lattices partially, stronger maps are to be
considered latter, incorporating the inf-maps.

Given a lattice $L := (L, \leq)$, where $X \subseteq L$, let us
recall some definitions from \cite[6.1.2, p430]{qtheory}.

\begin{definition}\label{def:sji} Let $\ell, m $ be elements of a lattice $L:= (L, \leq)$:
\begin{enumerate}
\ealph
    \item $\ell \in  L$  is \textbf{strictly join irreducible} (\sji) if whenever $ \ell = \bigvee
    X$
    there exists $x \in X$ such that $\ell = x$. \pSkip

    \item $\ell \in  L$  is \textbf{join irreducible} (\ji) if whenever $ \ell  \leq \bigvee
    X$
    there exists $x \in X$ such that $\ell \leq x$. \pSkip

    \item $m \in L$ is called \textbf{strictly meet irreducible} (\smi) if $m =
    \bigwedge X$ implies that there exists $x \in X$ such that $m =
    x$. \pSkip

    \item $m \in L$ is called \textbf{meet irreducible} (\mi) if $m
    \geq
    \bigwedge X$ implies that there exists $x \in X$ such that $m \geq x$.

\end{enumerate}
Join irreducibles are also called \textbf{primes}, while meet
irreducibles are called \textbf{co-primes}.
\end{definition}

For  a finite lattice $L := (L, \leq)$ it is easy to see that the
\smi 's that are not  $T$ (the top element) are the unique minimal
sets of meet generators of $L$, by universal algebra \cite{Mac}.
The top element $T$ is meet generated by the empty set.

We define $ \smiNoT (L) $ to be the number of \smi's not $T$,
i.e.,
$$ \smiNoT (L) := |\{ \ell \in L \ds | \ell \text{ is } \smi \neq T\} |,$$
and similarly define the number of \smi's
$$ \smiNo (L) := |\{ \ell \in L \ds | \ell \text{ is } \smi  \} |.$$
 Dually,  the $\sji$'s not $B$ (the bottom element) are the unique minimal subsets of
join generators of $L := (L,\leq)$, and we define their number to
be
$$ \sjiNoB (L) := |\{  \ell \in L \ds | \ell \text{ is  \sji} \neq B\} |,$$
and let
$$ \sjiNo (L) := |\{  \ell \in L \ds | \ell \text{ is  \sji} \} |.$$
\begin{example}\label{exmp:3.4.1}

Let $X_{2n}$ be a set of $2n$ elements, and let $\Lm_{2n}$ be the
semilattice whose elements are all subset $Y \subseteq X$ of
cardinality $ \geq n$ of $X_{2n}, $ together with the empty set
$\emptyset$, and the parietal order determined by inclusion. The
join of $\Lm_{2n}$ is set union and the determined  meet is set
intersection, unless its cardinality if is less than $n,$ which in
this case is made $\emptyset$.

%
%

It is easy to see that all \sji's not $\emptyset$ of $\Lm_{2n}$
are all the subsets of order $n$ of $X_{2n}$ and the \smi's are
all the subsets having $2n-1$ elements. Thus, we have the
following
$$ |\Lm_{2n}| = \frac{2^{2n} - \chos{2n}{2}}{2} + \chos{2n}{n} +1,  $$
where
$$ \sjiNo(\Lm_{2n}) = \chos{2n}{n} \dss{ \text{and}} \smiNo(\Lm_{2n}) = 2n. $$
Therefore $\sjiNo$ and $\smiNo$ can be differ exponentially.
\end{example}
As shown in Corollary \ref{cor:reversing}, the {\nook}  does not
changed under reversing the order. On the other hand, Example
\ref{exmp:3.4.1} shows that $\sjiNo$ and $\smiNo$ are changing
significantly; as $\sjiNo$ and $\smiNo$ are dual, they interchange
when reversing the order. Unlike the situation of \sji \ and \smi
\ for finite lattices, whose members can be differ, we will see
that the members of \ji's and \mi's are always equal.

\begin{lemma} \label{lem:4.5} Suppose $\ell_1, \dots, \ell_{j-1}, \ell_{j_1} \vee \ell_{j_2}, \ell_j , \dots,
\ell_k $ are independent in the lattice $(L, \leq)$. Then, for $i
=1 $ or $i=2$,  $\ell_1, \dots, \ell_{j-1}, \ell_{j_i}, \ell_j ,
\dots, \ell_m $ are independent.
\end{lemma}

\begin{proof} Write  $\ell_j := \ell_{j_1} \vee \ell_{j_2}$ and let
$U := \{ m_1, \dots m_k\}$ be the witness of $W := \{ \ell_1,
\dots \ell_k \} $ in $\cmp{A} := \cmp{A(L)}. $ Reordering the rows
of $\cmp{A}$, we may assume that the witness
$\rwcl{\cmp{A}}{U}{W}$ is of the form \eqref{eq:trgform}. Thus
$$m_s \not \leq \ell_t, \quad \text{for every }  1 \leq  s < t
\leq k,$$ and therefore $\ell_{j_1} \vee \ell_{j_2} = \ell_j  \not
\leq m_j$. So, either $\ell_{j_1}   \not \leq m_j$ or $\ell_{j_2}
\not \leq m_j$, say $\ell_{j_1}   \not \leq m_j$. But then $U = \{
m_1,\dots, m_k\}$ is also a  witness for $W_1 : =\{ \ell_1, \dots,
\ell_{j-1}, \ell_{j_1}, \ell_{j+1}, \dots \ell_k \}$ being
independent.
\end{proof}

\begin{proposition}\label{prop:4.4} $ \nk {L} \leq \sjiNo (L)$ for any finite lattice $L := (L ,
\leq)$.
\end{proposition}

\begin{proof} Assume  $m := \nk {L}$, and let $\ell_1, \dots, \ell_m
$ be independent.
 Since the \sji's not $B$ are join generate ~ $L$, each
$\ell_j$ can be written as $\ell_j = \bigvee_ { k =1}^{m}
\ell_{j_k}$ where $\ell_{j_k}$ are \sji's not $B$. Applying Lemma
\ref{lem:4.5} inductively, we obtain an independent subset
$\ell_1', \dots, \ell_m' $ where each $\ell'_j$ is $\sji \neq B$.
(This is a stronger statement than Proposition \ref{prop:4.4}.)
\end{proof}

\begin{corollary}\label{cor:4.7}
$L := (L, \leq)$ has  an independent subset $X$ of maximal
cardinality, .i.e.,
 $|X| = \nk{L}$, which is contained in $\{ \ell \in L \ds | \ell \text{ is
} \sji \neq B \}.$
\end{corollary}

%

\subsection{Spec of finite lattices}
The previous section leads us to a spectral theory of finite
lattices, see \cite[\S6-\S7]{qtheory}. Following Marshal, Stone,
and others, the basic approach of spectral theory of  finite
lattices is to consider the maximal distributive lattice generated
by the set of surmorphisms. For this purpose, we need more
structure.

Let $(\bool, \leq)$, $\bool:= \{0,1\}$, be the two element lattice
with the standard partial order. Given an element $\ell \in L$, $L
:= (L, \leq)$ a lattice, we define the lattice map
\begin{equation}\label{eq:sup-map}
 \lvrp_{\ell}: (L, \leq ) \onto (\bool, \leq),  \qquad \lvrp_{\ell} : x \mapsto  \left \{
\begin{array}{lll}
  1 &  & x \leq \ell, \\
  0 &  & \text{else.} \\
\end{array}
\right.
\end{equation}  The map $\lvrp_{\ell}$ is a sup-inf-map onto
$(\bool, \leq)$ iff $\ell$ is \mi (cf. Definition \ref{def:sji}),
and $\ell \neq T$. Conversely, a map  $\vrp : (L, \leq ) \onto
(\bool, \leq )$ is sup-inf-map onto $\bool$ iff $ \bigvee \ivrp(0)
= \ell$ with $\ell \neq T$ an \mi \ and $\vrp = \lvrp$.

The dual result also holds, namely the map
\begin{equation}\label{eq:inf-map}
 \gvrp_{\ell}: (L, \leq ) \onto (\bool, \leq), \qquad \gvrp_{\ell}
: x \mapsto  \left \{
\begin{array}{lll}
  1 &  & x \geq \ell, \\
  0 &  & \text{else.} \\
\end{array}
\right.
\end{equation} is a sup-inf-map onto $(\bool, \leq)$ iff
$\ell$ is \ji, and $\ell \neq B$. Conversely, a map  $\vrp : (L,
\leq ) \onto (\bool, \leq )$ is sup-inf-map onto $\bool$ iff $
\bigwedge \ivrp(1) = \ell$ with $\ell \neq B$ a \ji \ and $\vrp =
\gvrp$.

Given such a map $\vrp : (L, \leq ) \onto (\bool, \leq )$ as
above, we have a 1:1 correspondence $$ \bigvee \ivrp(0) \dss
\leftrightarrow \bigwedge \ivrp(1)$$ between \mi's not $T$ and
\ji's not $B$. Thus the number of \mi's not $T$ equals that of
\ji's not $B$, we denote this number $\miNoT$ and define
\begin{equation}\label{eq:s-map} \ss(L) : = \miNoT(L) -1 = |\{ \ell \in L \ds | \ell \text{ is } \mi \neq T
\}| -1.
\end{equation}

We consider the spec lattice morphism for a finite lattice $(L ,
\leq) : $
\begin{equation}\label{eq:spec}
    \spec(L): (L , \leq) \ds \to (\bool, \leq)^{\ss(L)}
\end{equation}
with
$$ \spec(L) =   \Dl \bigg( \bigotimes_{\ell \text{ is \mi } } \lvrp_\ell\bigg) =
\Dl \bigg(   \bigotimes_{\ell \text{ is \ji } } \gvrp_\ell \bigg)
$$ where $\Dl$ is the diagonal map. (This equality derived by the
above discussion.)

The map $\spec$ is  a sup-inf lattice morphism of $(L, \leq)$ onto
the finite distributive lattice $(\bool, \leq)^{\ss(L)}$, which is
isomorphic to the all subsets of a set of cardinality $\ss(L)$
under inclusion. Since subsets of a distributive lattice are
closed under meet and join containing $B$ and $T$, then clearly
they are distributive as well.

The image $\spec(L)$ of $\spec$ is a distributive lattice, and we
aim to show that it is the maximum  distributive image of a
sup-inf map of  $(L, \leq)$.
%
%
Using the notation of Definition \ref{def:upsetPos}, we deduce
directly from the definition of $\spec$ that for any $\ell, \ell'
\in L$
$$ \spec(\ell) = \spec(\ell') \dss
\iff \pup{\ell} \cap \{ \mi \neq T\} = \pup{\ell'} \cap \{ \mi
\neq T\} \dss \iff \pdn{\ell} \cap \{ \ji \neq B \} = \pdn{\ell'}
\cap \{ \ji \neq B \}. $$

Another important categorical notion is the Adjoint concept, known
also as ``Galois connection'', for finite lattices which is as
follows:

\begin{proposition}[Adjoint of sup-maps]\label{prop:ajoint}
Let $(L, \leq)$ and $(L', \leq)$ be two finite lattices. Assume
that $\vrp: (L, \leq) \to (L', \leq)$ is a sup-map, and let $\psi:
(L', \leq) \to (L, \leq)$ be the map (denoted also as
$\adjsup{\vrp}$) defined by
\begin{equation}\label{eq:gg}
\psi (\ell') := \adjsup{\vrp} (\ell') : = \bigvee \ivrp(\ell')
\end{equation} for each $\ell' \in L'$.
Then, the following properties hold:
\begin{enumerate}
    \item $\psi$ is an inf-map, \pSkip
    \item $\vrp(\ell) \leq \ell' \ds \iff \ell \leq \psi (\ell')$,
    \pSkip

    \item $\psi \circ  \vrp \circ \psi = \vrp $ and  $\vrp \circ \psi \circ \vrp = \vrp
    ,$ \pSkip

    \item $\psi$ is injective iff $\vrp$ is surjective, $\vrp$ is
    injective iff $\psi$ is surjective, \pSkip

    \item $\vrp(\ell) = \bigwedge \ipsi(\ell)$.
\end{enumerate}
\end{proposition}

\begin{proof}
 The proof is straightforward.
\end{proof}

In general, for a finite lattice,  \mi \ implies \smi, but the
converse holds only for a finite distributive lattice.

\begin{proposition}\label{prop3.3.4} The following are equivalent for a finite
lattice $(L, \leq)$:
\begin{enumerate}
    \item $(L, \leq)$  is distributive, \pSkip
    \item $\mi \iff \smi$ (resp. $\ji \iff \sji$), \pSkip
    \item $\spec$ is 1:1, \pSkip
    \item $\spec$ is a lattice isomorphism of $(L, \leq)$ and
    $\spec(L)$, with the induced order of $(\bool,
    \leq)^{\ss(L)}$.
\end{enumerate}

\end{proposition}

\begin{proof}$(1) \Rightarrow (2)$: If $p$ is \smi \ and $p \geq a \wedge
b$ then, by distributivity,
$$ p = p \vee p  = (a \wedge b ) \vee p  = (a \vee p) \wedge  ( b \vee
p).$$ Then by \smi \ $p = a \vee p $, say  on $p \geq a$, and thus
$p$ is \mi. \pSkip

$(2) \Rightarrow (3)$: The \smi's meet generate, so the \mi's
$\ell$ meet generate $L$. Thus, if $\ell_1, \ell _2 \in L $, where
$\ell_1 \neq  \ell _2$, then there exits an \mi \ $m$  such that
$\ell_1 \leq m$ and $\ell_2 \not \leq m$ (since each $\ell$ is the
meet of all \mi's $\geq \ell$). Thus, $m$ is different on $\ell_1$
and $\ell_2$, and hence $\spec$ is 1:1. \pSkip

$(3) \Rightarrow (4)$: Set $\vrp:= \spec$. Since $\vrp$ is a
bijective sup-map, by Proposition \ref{prop:ajoint}, the adjoint
map $\psi : \spec(L) \onto L$ exists and it is a bijective
inf-map. To complete this part we need to show that $\vrp(\ell_1)
\leq \vrp(\ell_2)$ implies $\ell_1 \leq \ell_2$. But, by
Proposition \ref{prop:ajoint}, $\vrp(\ell_1) \leq \vrp(\ell_2)$
implies $(\psi \circ \vrp )(\ell_2) \geq \ell_2$ and since $\vrp$
and $\psi$ are bijections, $\ell_2 = (\psi \circ \vrp) (\ell_2)$,
by definition of $\psi$. \pSkip

$(4) \Rightarrow (1)$ As already remarked, $\spec(L)$ is a
distributive lattice, and so is $(L, \leq)$ by isomorphism.
\end{proof}

\begin{corollary}[Birkhoff] $(L, \leq)$ is a distributive lattice
iff is isomorphic to a collection of subsets of a finite set $Z$
of cardinality $\ss(L)$, closed under set theoretic union and
intersection, including $\emptyset$ and whole~$Z$.
\end{corollary}

\begin{corollary} $\spec(L)$ is the unique maximal sup-inf image
of $L := (L , \leq)$ which is a distributive lattice. $\spec(L)$
is generated by $\ss(L) $ elements, which is also the number of
\smi's = \mi's not $T$ (also is the number of $\sji's = \ji's$ not
$B$).
\end{corollary}

\begin{proof} $\spec(L)$ is a distributive lattice, image of
sup-inf map of $(L , \leq).$  To see the unique maximality,
applying  Proposition \ref{prop3.3.4}.(3) to any such image
$\tlL$,
$$\xymatrix{
L \ar @{->>}[r]^\psi  & \tlL \ar@{->>}[r]^\spec_{\iso} &
\spec(\tlL),
 }$$
we see that $(\spec \circ \; \psi) (L) $ factors through $L \onto
\spec_L(L)$, and the rest of the proof follows from Proposition~
\ref{prop3.3.4}.
\end{proof}

\begin{remark}\label{rmk.3.4.7} If we endow $\{0,1 \}$ with the
Sierp\'{i}nski topology (not $T_2$) in which the closed sets are
$\{ \emptyset, \{ 0\}, \{0,1 \} \} $, then the null-kernel
topology is induced topology of the poset topology.  Thus, the
closed sets of $\spec(L)$ are of the form $V (\ell)$, for $\ell
\in L$,  with
$$V(\ell) := \{ m \in \spec(L) \ds | m \geq \ell\}. $$
See \cite[\S7]{qtheory}.
\end{remark}

\begin{example}\label{exmp.3.4.7}
Let $L := (L,\leq)$ be the following  lattice, $|L| = 5$,
\begin{equation}\label{eq:poset} \footnotesize   {\xymatrix{
 & & T \ar@{-}[dr] \ar@{-}[dl]& \\ & 2 \ar@{-}[d] & & 3 \ar@{-}[ddl]\\
& 1 \ar@{-}[dr] & & \\
&  & B &
 }}
\end{equation}
The \mi's not $T$ are $2$  and $3$, so the relations on $L$ given
by $\spec$ are the singletons and $\{1,2\}$.
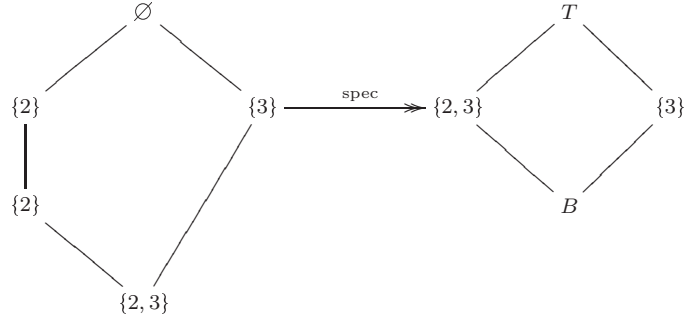
 $$  \footnotesize  {\xymatrix{
 & & \emptyset  \ar@{-}[dr] \ar@{-}[dl]& 
 & & & T \ar@{-}[dr] \ar@{-}[dl] &
 \\ & \{ 2 \} \ar@{-}[d] & & \{ 3 \} \ar@{-}[ddl] \ar@{->>}[rr]^{\spec} &   & \{2 ,3\} & & \{3\}\\
& \{ 2\} \ar@{-}[dr] & & & & & B \ar@{-}[ur] \ar@{-}[ul] &\\
&  & \{ 2,3\} &
 }}$$
Thus, for $L$ we have $\sjiNoB(L) = \smiNoT (L) = 2$, $\jiNoB(L) =
3$, and $\ss(L) = 2.$
\end{example}

\begin{example} Consider the semilattice $\Lm_{2n}$ in Example
\ref{exmp:3.4.1}. Then, $\ss(\Lm_{2n}) = 0 $, $\spec( \Lm_n) = \{
0 \} $.

\end{example}
\section{Dimension and related functions for finite lattices}

In this section we develop the theory of {\nook} of  finite
lattices, along with methods of computation and relations to other
lattice functions studied earlier.


\subsection{Duality of boolean modules}\label{ssec:4.2} (See \cite[Propositions
9.1.12-13]{qtheory}.) Given a finite  $\bool$-module $M := M
(\bool, +),$ we define the dual module $$ M^* := \{ \vrp: M  \to
\bool \ds | \vrp \text{ is a sup-map} \} ,$$ consisting of all
sup-maps over $M$.
Then, $M^*$ itself is a $\bool$-module closed under
addition, i.e., satisfying $(f_1 + f_2)(m) = f_1(m) + f_2(m)$ for
every $f_1, f_2 \in M^*$, whose unit is the constant zero mapping
$f_0 : m \mapsto 0$ for every $m \in M$.

\begin{remark}\label{rmk:4.2.a}
Clearly, for any $\ell \in L$, the sup-map $\lvrp_{\ell} : (L
,\leq) \to (\bool, \leq)$, cf.  Eq. \eqref{eq:sup-map}, is
contained in $\dlM$ and the map $\chi: \ell \mapsto \lvrp_{\ell}$
is bijection, as is easy to prove. \end{remark}

Having this view, we deduce that if the $\bool$-module $M$ is
considered as a lattice $(M , \leq)$, then the dual module $\dlM$
is the lattice $\rvs{M} := (M, \geq)$ obtained by reversing the
order of $M$ (realized as a lattice).

Similarly to  the $\spec $ map \eqref{eq:spec}, given for
lattices, for a finite $\bool$-module $L := (M, \leq)$, realized
as a lattice, we define the map
\begin{equation}\label{eq:} \cc_L :=
 \Dl \bigg( \bigotimes_{\ell \in L} \lvrp_{\ell} \bigg):  (L ,\leq)
    \ds  \to  \bool^{|L|},
\end{equation}
a sup embedding of $L$ into $\bool^{|L|}$. So as in the spec case,
$L$ is order isomorphic to $\cc(L)$ which is a subset of
$\bool^{|L|}$ closed under all joins. (Note that the meets of
$\cc(L)$ need not be the meets of $\bool^{|L|}$.)

Given the boolean module $M := M(\bool, + )$, and consider $\cc(M)
\subseteq \bool^{|L|}$ realized as an $|L| \times |L|$ matrix~ $C$
with coefficients in $\bool$ and whose rows are the members of
$\cc(M)$. Then,  one could define $\nk{M}$, the \nook \ of $M$, to
be the matrix rank $\rnk{C}$. This is the same  as in Definition
\ref{def:nook}, since $\cmp{A}$ \emph{equals the matrix}~ $C$.
Thus, the map $A \to \cmp{A}$  for the adjacency matrix  $A:=
A(L)$ of $(L, \leq) $ is given by passing from $L$ to $\cc(L)$ as
in Remark \ref{rmk:4.2.a}. This is the same view as in the
important \cite{IJmatII}.

\subsection{Pullback and push of independent subsets} The following notion
provides an important property of independence of subsets of
lattices.

\begin{definition}
Given a lattice map $\vrp: (L, \leq) \to (L', \leq)$,  we say that
an element $\ell \in L$ is a \textbf{pullback} of $\ell' \in L'$
if $\vrp(\ell) = \ell'$. \end{definition}

\begin{lemma} \label{lem:4.8.a} Let $(L, \leq)$ and $(L', \leq)$
be finite lattices and let $\vrp: (L, \leq) \onto (L', \leq)$ be
a  sup-map of~ $L$ to~ $L'$. Assume $\ell'_1, \dots, \ell'_k$ are
independent elements of~ $L'$ and,  taking one representative
$\ell_i$ for each $\vrp^{-1}(\ell'_i)$, let $\ell_1, \dots, \ell_k
\in L$ be their pullbacks. Then, $\ell_1, \dots, \ell_k$ are
independent in~ $L$.
\end{lemma}

\begin{proof} Let $A := A(L) $  and $A' := A(L')$  be respectively
the matrix structure of $L$ and $L'$. Let $W' = \{ \ell'_1, \dots,
\ell'_k \} \subseteq L'$, and let $\rwcl{\cmp{A'}}{U'}{W'}$ be a
witness of $W'$ for some $U' := \{ m'_1, \dots, m'_k \}$. Assume
$\vrp$ is as in Proposition \ref{prop:ajoint} and let
$$\psi := \adjsup{\vrp} : (L', \leq) \ds \to (L, \leq) ,$$ cf. Eq.
\eqref{eq:gg}. Define $m_i := \psi(m'_i)$ for each $i = 1,\dots,
k$, and let $U := \{m_1, \dots, m_k \} $. We claim that
$\rwcl{\cmp{A}}{U}{W}$ is a witness of $W := \{ \ell_1, \dots,
\ell_k\} $ in $L. $ Indeed, permuting the columns of $\cmp{A}$, by
Lemma \ref{lem:2.1.f}, we may assume that $\rwcl{\cmp{A}}{U}{W}$
is of the Form \eqref{eq:trgform}.

If $\ell_j \leq m_j$, then applying $\vrp$ -- a sup-map -- we get
$\ell'_j \leq m'_j$ which is  false. Thus, $ \ell_j \nleqslant
m_j$, and we need to show that $\ell_i \leq m_j$ for $i < j$. But
$\vrp(\ell_i) = \ell'_i \leq \vrp(m_j) = m'_j$,    by Proposition
\ref{prop3.3.4}.
\end{proof}

\begin{corollary}\label{cor:4.8.b} $ $
\begin{enumerate} \eroman
    \item If $\vrp: (L, \leq) \onto (L', \leq)$ is an onto sup-map,
    then $\nk{L} \geq \nk{L'},$ \pSkip

    \item If $(L', \leq)$ is a sub-module of  $(L, \leq)$, i.e., a
    subset of $L'$ closed under join, then  $\nk{L} \geq \nk{L'}.$
\end{enumerate}

\end{corollary}

\begin{proof} (i): Follows from Lemma \ref{lem:4.8.a}.
\pSkip  (ii): Immediate by part (i).
\end{proof}

We say that a finite $\bool$-module $M'$ \textbf{divides} a
$\bool$-module $M$, written $M' < M$, iff $M'$ is the image of a
sup-map of a sub-module of $M$. Accordingly, $M' < M$ implies
$\nk{M'} < \nk{M}.$

\begin{theorem}\label{thm:4.9} For any finite lattice $L = (L,
\leq)$ we have the equality  $\nk{L} = \hgt{L}$.
\end{theorem}
\begin{proof} If $0 < \ell_1 < \cdots < \ell_k$ is a chain in $L$,
then $ \ell_1, \dots, \ell_k$ are independent with witness $U :=
\{ m_1, \dots, m_k\}$, $m_1 = B$, $m_2 = \ell_1, \dots, m_k =
\ell_{k-1}$. Thus, $\nk{L} \leq k \leq \hgt{L}$.

Suppose $W := \{ \ell_1, \dots, \ell_m \} $ are independent with a
witness $\rwcl{\cmp{A}}{U}{W}$, $ U:= \{ m_1, \dots, m_k \}$, of
the Form~\eqref{eq:trgform}. Then the chain
\begin{equation}\label{eq:4.9.a}
    m_1 \wedge m_2 \wedge \cdots \wedge m_k \ds \leq m_2 \wedge m_3 \wedge \ds \cdots \wedge m_k \ds \leq
    \cdots  \ds \leq m_{ k-1 }  \wedge m_k \ds \leq m_k \ds \leq T
\end{equation}
is a strict chain in $L$, since $\ell_1, \dots, \ell_{k-1} \leq
m_k$, $\ell_k \not \leq m_k$ then
\begin{equation}\label{eq:4.9.b}
\begin{array}{rclcrcl}
    \ell_1,  \dots, \ell_{k-2}  & \leq &  m_{k-1} \wedge m_k, & \quad &
    \ell_{k-1},
    \ell_k  & \not \leq &  m_{k-1} \wedge m_k, \\[2mm]
    \ell_1, \dots, \ell_{k-3}  & \leq &  m_{k-2} \wedge m_{k-1} \wedge m_k, & \quad &
    \ell_{k-2}, \ell_{k-1},
    \ell_k  & \not \leq &  m_{k-1} \wedge m_k, \\[2mm]
    \qquad \vdots & &  \qquad \vdots & &  \vdots \qquad & &  \qquad \vdots \\[2mm]
      \ell_1 & \leq &  m_2 \wedge \cdots \wedge m_{k-1} \wedge m_k,
      & \quad &
    \ell_2, \dots,  \ell_{k-1},
    \ell_k  & \not \leq  & m_2 \wedge \cdots \wedge m_{k-1} \wedge m_k.
\end{array}
\end{equation}
and $\ell_1 \not \leq  m_1 \wedge \cdots \wedge m_{k}$. Thus,
$\nk{L} \leq k \leq \hgt{L}.$
\end{proof}

\begin{remark} By Theorem \ref{thm:4.9},  we see that $\bool^{(n)} = \bool \oplus \cdots  \oplus \bool$ has rank
$n$
\end{remark}



Theorem \ref{eq:4.9.a} shows how to compute the \nook \ of a given
lattice, but we also want a way to compute independent subsets. To
do so we need the following notion:
\begin{definition}\label{def:push}
Let $L:= (L, \leq)$ be a finite lattice, and let $U := \{ m_1,
\dots, m_k \}$ be a witness of $W : = \{ \ell_1, \dots, \ell_k \}
$. A subset  $\tlW : = \{ \tlell_1, \dots, \tlell_k \} \subseteq
L$ with $\tlell_i \leq \ell_i$ and $\tlell_i \not \leq m_i$ for
every $i = 1, \dots, k$ is called a \textbf{push} of $W$ with
respect to $U$.
\end{definition}
\begin{proposition}[``Pushing''] \label{prop:4.10.b} A push of an
independent subset $W$ with witness $U$ is independent  with the
same witness.
\end{proposition}
\begin{proof}
Clear, since $U$ is a witness of $\tlW$ as well.
\end{proof}

\begin{proposition}\label{prop:4.10.a} Let $L := (L, \leq)$ be  a
finite lattice. The independent subsets of $L$  are exactly pushes
of chains of $L$. That is, if $W := \{ \ell_1, \dots, \ell_k\}$
are independent with  witness $U: =\{ m_1, \dots, m_k\}$, then, is
an the proof Theorem \ref{thm:4.9},
$$ \tlm_1 <  \tlm_2 <  \cdots <  \tlm_k < T, \qquad \tlm_j = m_j \wedge \cdots \wedge
m_k,
$$
is a strict chain. So
$$ \tlell_ 2 <  \tlell_3 <  \cdots <  \tlell_k <  T,$$ is an independent
set with witness $\tlm_1, \dots, \tlm_k$, and $\ell_1, \dots,
\ell_k$ is a push of this chain.
\end{proposition}
\begin{proof}
Clear by construction.
\end{proof}

\subsection{Lattice completion of finite posets} Given a poset $P := (P,
\leq)$, let \begin{equation}\label{eq:upSet}
 \pdn{P} := \{ \pdn{p} \ds | p \in P \},
\end{equation}  where $\pdn{p}$ is the down-set of
$p$, cf. Definition \ref{def:upsetPos}. Then,
\begin{equation}\label{eq:Hasse}
\Hs{P} := (\pdn{P}, \subseteq) \end{equation} is the \textbf{Hasse
diagram} of $P$, a poset is by itself,  whose partial order is
determined by inclusion.

We define  $\overline{\pdn{P} }$ to be the closure intersection of
all subsets of $\pdn{P}$ including the empty set. Clearly
$\overline{\pdn{P} }$ contains the top element $T = P$. Then,
$\overline{\pdn{P} }$ is the \textbf{Dedekind-MacNeille
completion} of $P$, denoted also as $\DM{P}$, and it is a finite
complete lattice with meet set intersection, top element $P$, and
determined join. Moreover $P$ is order embedded into $\DM{P}$ by
$$\Gm: P \to \DM{P}, \qquad \Gm : p \mapsto \pdn{p}.$$

Similarly, we close all the subsets of $\pdn{P}$ under union
(including the empty set), and denote this \textbf{union closure}
as $\UC{P}$ -- a finite complete lattice. This is a lattice
completion of the poset $P$. The order ideals of $P$ with joint
set union, bottom element $\emptyset$ and determined meet (which
is just set intersection) shows that $\UC{P}$ is a ring set, where
$\Phi: p \mapsto \pdn{p}$ is an order embedding $\Phi: P
\hookrightarrow \UC{P}$, see \cite{Lattice,Gehrke}.

In some reasonable precise sense  $\DM{P}$ is the smallest lattice
completion of the poset  $P$, and $\UC{P}$ is the largest lattice
completion of $P$.

\begin{remark}\label{rmk:DM-Lattice}
The Dedekind-MacNeille completion of a finite lattice  $L :=
(L,\leq)$ is a lattice isomorphic to $L$ (see
\cite{Lattice,Gehrke}).
\end{remark}

\begin{example} Let $P := (P, \leq)$ be the $6$-element poset
$P :=\{a,b,c,d,e,f\}$ whose Hasse diagram  is
 \vskip -1cm  
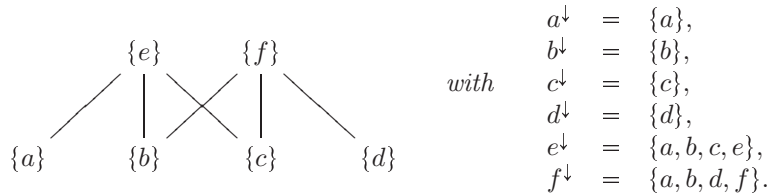
$$\xymatrix{
 & \{ e \}  & \{  f \} &  \\
 \{ a \} \ar@{-}[ur]  & \{b \}  \ar@{-}[u] \ar@{-}[ur]  &  \{ c \} \ar@{-}[u]  \ar@{-}[lu] & \{  d \} \ar@{-}[lu] \\
 } \quad
\begin{array}{llll} \\ \\ \\
  & \pdn{a} & = & \{a \}, \\
&   \pdn{b} & = & \{b \}, \\
\text{with} \quad   & \pdn{c}  & = & \{c \}, \\
 &  \pdn{d} & = & \{d \}, \\
 &  \pdn{e} & = & \{a,b,c,e \}, \\
 &  \pdn{f} & = & \{a,b,d,f \}. \\
\end{array}
 $$
Computing the matrix $A := A(P)$, providing  $\cmp{A}$,  we get
$$
A = \begin{array}{c|cccccc}
      \leq & a & b & c & d & e & f \\\hline
      a & 1 & 0 & 0 & 0 & 1 & 0\\
      b & 0 & 1 & 0 & 0 & 1 & 1 \\
      c & 0 & 0 & 1 & 0 & 1 & 1\\
      d & 0 & 0 & 0 & 1 & 0 & 1\\
      e & 0 & 0 & 0 & 0 & 1 & 0 \\
      f & 0 & 0 & 0 & 0 & 0 & 1 \\

    \end{array} \dss \Rightarrow
\cmp{A} = \begin{array}{c|cccccc}
      \not \leq & a & b & c & d & e & f \\\hline
      a & 0 & 1 & 1 & 1 & 0 & 1  \\
      b & 1 & 0 & 1 & 1 & 0 & 0\\
      c & 1 & 1 & 0 & 1 & 0 & 0\\
      d & 1 & 1 & 1 & 0 & 1 & 0\\
      e & 1 & 1 & 1 & 1 & 0 & 1\\
      f & 1 & 1 & 1 & 1 & 1 & 0 \\
    \end{array}
$$
which shows that $\nk{P} = \rnk{\cmp{A}} = 4$. The
Dedekind-MacNeille completion $\DM{P}$  of $P$  is then
$$ 
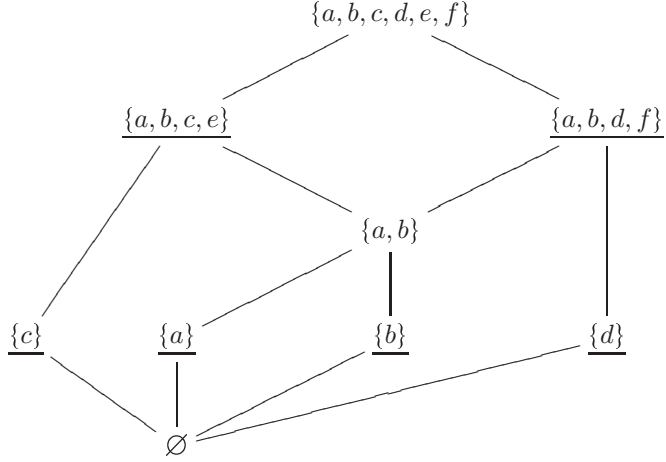
\xymatrix{ & & \{ a,b,c,d,e,f \}    \ar@{-}[dl]  \ar@{-}[dr] & \\
 &  \ul{\st{a, b,c,e}} & & \ul{\st{a, b,d,f}}& \\
 &   & \ar@{-}[ul] \st{a,b} \ar@{-}[ur] & & \\
 \ul{\st{c}} \ar@{-}[uur]  & \ul{ \st{a} }   \ar@{-}[ur]
  &\ul{ \st{b} }  \ar@{-}[u] &\ul{ \st{ d }  \ar@{-}[uu] }  \\
 &   \ar@{-}[ul] \ar@{-}[u]\emptyset \ar@{-}[ur] \ar@{-}[urr] &&  & \\
 }$$
which is a lattice of height $4$, and is an order embedding $P
\hookrightarrow \DM{P}$ of $P$ into $\DM{P}$. (The image of $P$ in
$\DM{P}$ is indicated by underlines.)

\end{example}

\begin{theorem}
Let $\DM{P}$ be the Dedekind-MacNeille completion of a poset $P:=
(P,\leq)$, with the order embedding $\Gm: P \hookrightarrow
\DM{P}$. Abusing notation, we assume that $\Gm$ is the identity
map, i.e., realized as  $P \subseteq \DM{P} $.
\begin{enumerate} \eroman
    \item $\nk{P} = \nk{\DM{P}} = \hgt{P}.$ \pSkip
    \item The independent subsets of $P$ are those of  $\DM{P}$
    restricted to $P$.
\end{enumerate}
\end{theorem}

\begin{proof} The proof follows the arguments of the proof of Theorem
\ref{thm:4.9}. \pSkip

(i): Suppose $\hgt{\DM{P}} = k$ and let
$$ \emptyset =  P_0 \varsubsetneq P_1 \varsubsetneq \cdots \varsubsetneq P_{k-1 } \varsubsetneq P_k  = P $$
be a maximal chain of $P$. Define
$$ Q_{i} := \bigcap \big \{ \pdn{p} \ds | p \in P_i \big \}, \qquad i=1, \dots, k,$$
to get a maximal  chain
$$ B = Q_k \varsubsetneq  Q_{k-1}   \varsubsetneq  \cdots \varsubsetneq  Q_{1} \varsubsetneq  Q_{0} = T$$
in $\DM{P}.$ For every  $i =1,  \dots k$, pick $$q_i \in Q_{i}
\setminus Q_{i+1}.$$

Pick $s_1 \in Q_1$ so that $s_1 \geq q_1$ and $s_1 \ngeq q_{0}$,
repeat this process recursively, picking $s_i \in Q_i$ such that
$s_i \geq q_i$ and $s_i \ngeq q_{i-1}$. Accordingly, for each
$i=1,\dots n$ we have
$$ s_i \geq q_i, \dots , q_k \dss {\text{and}} s_i \ngeq q_{i-1}.$$
This means that $s_1,\dots,s_k$ provide a witness for the bottom
element, i.e., it is of \nook \ $k$. \pSkip

(ii): Returning to the proof of Proposition \ref{prop:4.10.a},
where if $\ell_1, \dots, \ell_k$ in the lattice $\DM{P}$  are
independent with witness $m_1, \dots, m_k$, then setting $$\tlm_j
= m_j \wedge \cdots \wedge m_k
$$
so that
$$ \tlm_1 < \cdots < \tlm_k < T $$ are an independent subset with
witness $\tlm_1, \dots  , \tlm_k$, and $\ell_1, \dots, \ell_k$ is
a push of this chain.

Then, similar  to the proof  of part (i), when $q_1, \dots, q_k
\in P $ are points of $P$, we can find $\{s_1, \dots, s_k
\}\subseteq P $ that is a witness for  $\{q_1, \dots, q_k\}
\subseteq   P $.
\end{proof}

\subsection{Reformulation of ``pushing-chains'' to obtain
independent subsets of finite posets}\label{ssec:4.4}

Let $\Hs{L} := (\pdn{L}, \subseteq)$ be the Hasse diagram of a
finite lattice $L := (L,\leq)$, cf. \eqref{eq:Hasse}. Assign to
each edge $(\pdn{p_i}, \pdn{p_{i-1}})$ of $\Hs{L}$,  recording the
relation $\pdn{p_i} \subset \pdn{p_{i-1}}$, the set theoretic
difference
$$Q_{i} := \pdn{p_{i-1}} \ds \sm \pdn{p_i}.$$ Then, given a strict maximal chain
$$\pdn{T} \ds = \pdn{p_0} \ds >  \pdn{p_1} \ds > \cdots \ds
>  \pdn{p_{k-1}} \ds  > \pdn{p_k} = B$$ of $L$ from top to
bottom in $L$,   these $Q_i$ are disjoint and their union equals
$L \sm \{ B \}$.


We call the collection
$$ \tQ := Q_1, \ds \dots, Q_{k}$$ a \textbf{partition} of $L$. Note that
 these partitions correspond to different chains of $L$ and thus could have different
 lengths.

\begin{definition}\label{def:pcs}
A subset $X \subseteq L$ is a \textbf{partial cross section} of a
partition $\tQ$ iff each $ x \in X$ lies in a distinct $Q_i$,
i.e., $|X \cap Q_i| \leq 1$ for each $i = 1, \dots,k.$ (In such a
case, we also say that $X$ is an independent subset of $\tQ$.) A
\textbf{basis} of  a partition $\tQ$ is a partial cross section
$X$ of maximal cardinality.
\end{definition}

\begin{example}
Let $(L, \leq)$ be the finite lattice \eqref{eq:poset} as in
Example \ref{exmp.3.4.7}. Then, computing the Hasse diagram
$\Hs{L}$ and the differences along edges of maximal chains,  we
get
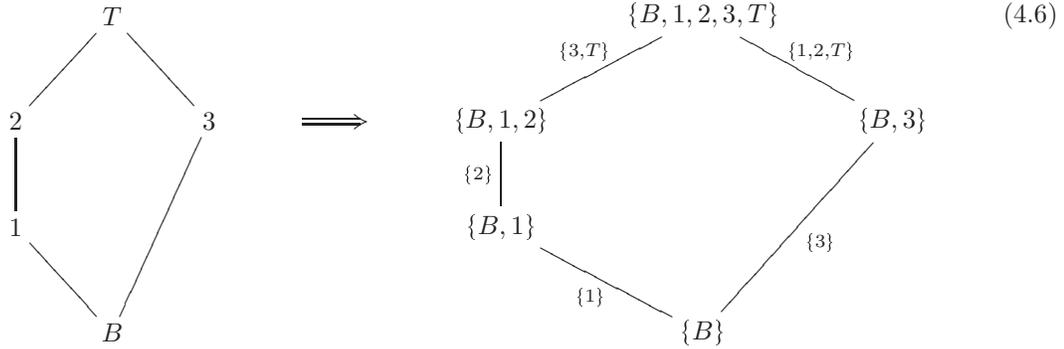
\begin{equation}\label{eq:partition}    {\xymatrix{
 & & T \ar@{-}[dr] \ar@{-}[dl]&
&  & &  & \{ B, 1,2,3, T \} \ar@{-}[dr]^{\{ 1,2, T \}} \ar@{-}[dl]_{\{ 3, T \}} \\
  & 2 \ar@{-}[d] & & 3 \ar@{-}[ddl]
  & \ar@{=>}[r] & & \{ B, 1,2 \}  \ar@{-}[d]_{\{2\}}  & & \{ B ,3 \} \ar@{-}[ddl]^{\{3\}} \\
& 1 \ar@{-}[dr] & &
& & & \{ B, 1 \} \ar@{-}[dr]_{\{ 1 \}}  & &\\
 &  & B & & & & & \{ B  \}
 }}
\end{equation}
where each chain determines a partition of $L \sm \{ B \} $, i.e.
$$ \{ \{1\}, \{ 2\}, \{ 3, T \} \} \dss{\text{and}} \{ \{3 \}, \{
1,2, T \} \} .$$   The bases of these partitions are therefore:
$$ \{ 1,2,  3\}, \quad \{1,  2,  T \}, \quad  \{ 3, T \},   $$
(clearly are not of the same cardinality).
\end{example}

\begin{theorem}\label{thm:pcs}  A subset $X$  of a lattice $L$ is independent iff $X$ is a partial cross section  for
the partition of $L$ corresponding to some set theoretic maximal
chain of $L$ from top to bottom.
\end{theorem}
\begin{proof}
  The proof is just a reformulation of
``pushing chains" argument for lattices,  Propositions
\ref{prop:4.10.b} and~ \ref{prop:4.10.a}.
 \end{proof}


\begin{iremark}
The same construction defined above works for finite posets  $P :=
(P,\leq)$ by considering the independent partitions of the
Dedekind-MacNeille completion $\DM{P}$. Then,  restricting these
partitions to $P$ and taking the partial cross sections give the
independent subsets of $P$. The proof is the same as before.
\end{iremark}

\section{Hereditary collections} Recall Definitions 2.1 and 2.5,
and the basic terminology,
 of \cite{IJmat}:

\begin{definition}\label{def:hereditary}
Let $E$  be a  set  and let $\tH \subseteq \Pow(E)$ be an nonempty
 collection of subsets $J$ of~$E$. The nonempty collection
$\tH$ is called \textbf{hereditary} if every subset $J'$ of any $J
\in \tH$ is also in $\tH$, more precisely: \boxtext{
\begin{enumerate} \eroman
    \item[HT1:] \ $\tH$ is nonempty, \pSkip

    \item[HT2:]  \ $J' \subseteq J$, $J \in \tH \ \imp \ J' \in
    \tH$.
\end{enumerate}}
(Hence, the empty set $\emptyset$ is also in $\tH$.)
The pair $\H := (E,\tH)$, with $\tH$ hereditary over $E$, is
called a \textbf{hereditary collection}.
\end{definition}

A subset $J \in \tH$ is called \textbf{independent}; otherwise is
said to be \textbf{dependent}.  A minimal dependent subset (with
respect to inclusion)  of $E$ is called a \textbf{circuit}. A
single element $x \in E$ that forms a circuit of $\H := (E, \tH)$,
or equivalently it belongs to no basis, is called a \textbf{loop}.
Two elements $x$ and $y$ of $E$ are said to be \textbf{parallel},
written $x \prll y$, if the $2$-set $\{x, y \}$ is a circuit of
$\H$. A hereditary collection is called \textbf{simple} if it has
no circuits consisting of $1$ or $2$ elements, i.e.,  has no loops
and no parallel elements.

\begin{definition} We say that  $\H = (E,\tH)$ satisfies the \textbf{point
replacement property} iff
 \boxtext{
\begin{enumerate} \eroman
    \item[PR:]  For every  $\{ p\}  \in \tH$ and every nonempty subset $J \in \tH$ there exists
    $x \in J$ such that $J - x  + p \in \tH$.
\end{enumerate}}
\end{definition}

Given a hereditary collection $\H = (E,\tH)$ that  satisfies PR,
then $\H$ is simple iff all of its subsets of $2$ or less element
are independent. The proof is the same as for the matroid case.

\begin{theorem}[{\cite[Theorem 5.3]{IJmat}}]
A vector hereditary collection \cite[Definition 4.3]{IJmat}
determined by the columns of a boolean matrix satisfies the point
replacement property.
\end{theorem}



\begin{theorem}\label{thm:5.3}
If a simple hereditary collection $\H := (E, \tH)$  has a boolean
representation, then there exist partitions $\tQ_1, \dots,
\tQ_\ell$ of $E$ so that the members of  $\tH$ are the  partial
cross sections.
\end{theorem}


The statement of Theorem \ref{thm:5.3} can be strengthen to
Theorem \ref{thm:5.4}, basing on the construction as described
next.

Let $\H := (E, \tH)$  be a simple hereditary collection, and
assume it has  a boolean representation  $A := A(\H)$. Augment the
rows of $A$ by all possible rows having exactly one entry $0$ and
the others $1$; call this matrix $B$. Then augment the enlarged
matrix $B$ again by adding the sups of all possible row subsets,
and denote this new matrix by $A'$.

Define the ``closed sets'' $C : = \clos(A')$ of $A'$ by taking the
collection of row-subsets of $A'$ whose members have a $0$-entry
in $r$,  for each row $r$  of $A'$. (Denote such a row as
$r^{(0)}_j.$) Then $\clos(A')$ is closed under all intersections
(so it includes $E$ and the empty set $\emptyset$) and is also
given by closing $\clos(A)$ under all intersections. Thus, $C$ is
a lattice with meet intersection and determined join being $\clos(
X \cup Y )$, where closure of   $Z$,  a subset of subset $E$, is
the intersection of all members of $C$ containing $Z$.
\begin{theorem}\label{thm:5.4} Let $A'$ be a as constructed above for a simple hereditary collection $\H : = (E, \tH).$
\begin{enumerate} \ealph
    \item The rows of $A'$ from  a lattice $L' : = \Lat{A'}$,  under sup and determined
join, which is sup-generated by the rows of $A'$. \pSkip
\item The independent subsets of $A$ and $A'$ are the same  by  Lemma \ref{lem:4.5}.
\pSkip
\item
   The map given by  $r_i \mapsto r_i^{(0)}$ is a reverse
isomorphism of $\Lat{A'}$ and $\clos(A')$. \pSkip
  \item  The partial cross sections of the partitions of $\clos (A')$  give exactly $\tH$.

\end{enumerate}

\end{theorem}

\begin{proof}  (a) and (c) are clear, while (b) is obtained by  Lemma
\ref{lem:4.5}. \pSkip

(d):  Consider a set theoretic maximal chain
       $$\emptyset = C_k  \ds < C_{k - 1} \ds < \cdots \ds  < C_1 \ds <  C_0 =E$$
from  $\emptyset$ to $E$ in $\clos(A')$.
 Since each $C_j$
corresponds a row $r^{(0)}_j$, replacing $C_j$  by the
corresponding row $r^{(0)}_j$ and reversing the order of the chain
we obtain the following chain in $\Lat{A'}$:
\begin{equation}\label{eqq} [0 \cdots 0] =  r^{(0)}_0 < \cdots < r^{(0)}_j      < \cdots  <
r^{(0)}_k = [1 \cdots 1], \qquad k \leq n,
\end{equation} where $r^{(0)}_0$ is a row whose entries are all $0$ and
$r^{(0)}_k$ is a row whose entries are all $1$.

We number the element of $E$ as $e_i$, where $ i = 1,2,\dots, n$
and $n = |E|$.  By induction we can assume that each $r^{(0)}_j$
has all of its $1$-entries first on the left and then all
$1$-entries. Let each $r^{(0)}_j$ have its $1$-entries up to $i_j$
in $E$, and consider the partition $$\tQ :=  \{ 1,\dots, i_1 \} ,
\{ i_1 +  1 ,  \dots, i_2\}, \ds \dots, \{ i_{k-1}+1,\dots, i_k\},
\qquad i_k = n,
$$ of $E$. By Proposition \ref{prop:4.10.b}, Proposition
\ref{prop:4.10.a}, and \S\ref{ssec:4.4},
 we see that the partial cross section of  $\tQ$ are just the
pushes of the chain \eqref{eqq}. This proves (d).
\end{proof}

\begin{proposition}
Not any hereditary collection that satisfies PR (even if it turns
out to be isomorphic to its dual) has a boolean representation.
\end{proposition}

\begin{proof}
For example consider the hereditary collection $\H: = (E,\tH)$
with $E = \{1, 2, 3,4,5 \} $ whose bases are
\begin{equation}\label{eq:basses}B_1 : =  \{1, 2, 3\}, \quad B_2 : = \{1,
2,4\}, \quad B_3 : =  \{ 2, 3,5\}, \quad B_4 : =  \{1,4,5\}, \quad
B_5 : =  \{3,4,5\}.
\end{equation}
It is easy to check that $\H$ satisfies PR and  is isomorphic to
its dual $\H^*$ (cf. \cite[Definition 2.15]{IJmat}) whose bases
are the 3-subsets of $E$ excluding the bases of $\H.$

Since $B_1$ is a basis, $E$ has a partition $\tQ = Q_1, Q_2, Q_3 $
with $i \in Q_i$, with $4$ and $5$ belong to these subsets
$Q_i$'s. Since the bases are as given in \eqref{eq:basses}, we are
``enforced'' to have the partition \begin{equation}\label{eq:par}
Q_1 : =  \{1, 5\}, \quad Q_2 : = \{2\}, \quad Q_3 : =  \{ 3,4\}.
\end{equation}
But then, by Theorem \ref{thm:5.3},  $\{2,4,5 \}$ which is not a
basis is also independent  -- a contradiction. This means that
$\tH$ can not have a boolean representation, since it must then be
given be partial cross section of the partition, in particular
$\{1,2,3 \} $ must also be given in this  way for which only
\eqref{eq:par} works.
\end{proof}



\begin{thebibliography}{10}

\bibitem{Cameron}
\newblock P. J. Cameron.  Chamber systems and buildings, The Encyclopaedia
of Design Theory,  May 30, 2003.



\bibitem{Lattice} G. Gr\"{a}tzer.
\newblock {\em Lattice Theory: Foundation}, Verlag: Birkh\"{a}user,
2011.


\bibitem{Gehrke}
\newblock M.~Gehrke, R.~Jansana,
A.~Palmigiano. $\Delta_1$-completions of a poset,  preprint, March
2011.

\bibitem{zur05TropicalAlgebra}
Z.~Izhakian.
\newblock Tropical arithmetic and tropical matrix algebra.
\newblock {\em Communication  in Algebra}, 37(4):1–--24, 2009.

\bibitem{IzhakianTropicalRank}
Z.~Izhakian.
\newblock The tropical rank of a tropical matrix. preprint at arXiv:math.AC/060420, 2006.



\bibitem{IJmat}
Z.~Izhakian,  J.~ Rhodes.
\newblock New representations of matroids and generalizations.
\newblock Preprint at arXiv:1103.0503, 2011.

\bibitem{IJmatII}
Z.~Izhakian,  J.~ Rhodes.
\newblock  Boolean representations of matroids and lattices.
\newblock Preprint at arXiv:1108.1473, 2011.


\bibitem{IzhakianRowen2007SuperTropical}
Z.~Izhakian, L.~Rowen.
\newblock Supertropical algebra.
\newblock {\em Advances in Mathematics}, 225(8):2222–--2286, 2010.



\bibitem{IzhakianRowen2008Matrices}
Z.~Izhakian, L.~Rowen.
\newblock Supertropical matrix algebra.
\newblock {\em Israel Journal of Mathematics},  182(1):383--424, 2011.

\bibitem{IzhakianRowen2009Equations}
Z.~Izhakian, L.~Rowen.
\newblock Supertropical matrix agebra {II}: Solving tropical equations.
\newblock {\em Israel Journal of Mathematics}, 186(1):69-97, 2011.


\bibitem{IzhakianRowen2010MatricesIII}
Z.~Izhakian, L.~Rowen.
\newblock {Supertropical matrix algebra III: Powers of matrices and generalized eigenspaces.}
\newblock {\em Journal of Algebra}, 341(1):125--149, 2011.


\bibitem{IzhakianRowen2009TropicalRank}
Z.~Izhakian,  L.~Rowen.
\newblock The tropical rank of a tropical matrix.
\newblock {\em Communication   in Algebra}, 37(11):3912 -- 3927, 2009.

\bibitem{Mac}
S.~Mac Lane.
\newblock {\em Categories for the working mathematician }.
\newblock Springer, 1971.

\bibitem{Mar}
G. Markowski, \newblock{Primes, irreducibles and extremal}
lattices,
\newblock {\em Order}, 9:265-290, 1992

\bibitem{qtheory}
J.~Rhodes, B.~Steinberg.
\newblock {\em The q-theory of Finite Semigroups}.
\newblock Springer-Verlag, 2009.

\end{thebibliography}

\end{document}